\documentclass[12pt]{article}
\usepackage[centertags]{amsmath}
\usepackage{amsfonts}
\usepackage{amssymb}
\usepackage{latexsym}
\usepackage{amsthm}
\usepackage{newlfont}
\usepackage{graphicx}
\usepackage{listings}
\usepackage{booktabs}
\usepackage{abstract}
\date{}

\newlength{\defbaselineskip}
\newcommand{\setlinespacing}[1]%
           {\setlength{\baselineskip}{#1 \defbaselineskip}}

\newcommand{\N}{{\mathbb{N}}}

\newcommand{\actaqed}{\hfill $\actabox$}
{\medskip\noindent \textit{Proof of #1. }}%
{\actaqed \medskip}

\def\Di{{\mathcal D}}

\def\cB{\mathcal B}

\def\R{{\mathbb R}}
\def\Z{\mathbb Z}

\def \<{\langle}
\def\>{\rangle}

\def \La{\Lambda}
\def \ep{\epsilon}

\def \ff{\varphi}

\def\bt{\beta}

\def\la{\lambda}

\def \supp{\operatorname{supp}}
\def \sp{\operatorname{span}}

\def\ba{\mathbf a}
\def\bb{\mathbf b}
\def\bx{\mathbf x}
\def\by{\mathbf y}
\def\bz{\mathbf z}
\def\bk{\mathbf k}
\def\bu{\mathbf u}

\def\bm{\mathbf m}

\def\bp{\mathbf p}

\def\btt{\mathbf t}

\def\bW{\mathbf W}
\def\bH{\mathbf H}
\def\bB{\mathbf B}
\def\bK{\mathbf K}

\def\bt{\beta}

\newtheorem{Theorem}{Theorem}[section]
\newtheorem{Lemma}{Lemma}[section]
\newtheorem{Definition}{Definition}[section]
\newtheorem{Proposition}{Proposition}[section]
\newtheorem{Remark}{Remark}[section]

\newtheorem{Conjecture}{Conjecture}[section]
\numberwithin{equation}{section}

\newcommand{\be}{\begin{equation}}
\newcommand{\ee}{\end{equation}}

\begin{document}

\title{Remarks on numerical integration, discrepancy, and diaphony}
\author{V.N. Temlyakov\thanks{University of South Carolina and Steklov Institute of Mathematics.  }}
\maketitle
\begin{abstract}
{The goal of this paper is twofold. First, we present a unified way of formulating numerical integration problems from both approximation theory and discrepancy theory.   Second, we show how techniques, developed in approximation theory, work in proving lower bounds for recently developed new type of discrepancy -- 
the smooth discrepancy. }
\end{abstract}

\section{Introduction} 
\label{I} 

We study numerical integration. The goal is to obtain optimal rates of decay of errors 
of numerical integration for functions from a given function class. Theoretical aspects of the problem of numerical integration are intensely studied in approximation theory and in discrepancy theory. It is known (see, for instance, \cite{VT89}) that the problem of optimal error of numerical integration of functions with mixed smoothness and the problem of minimal discrepancy of point sets of fixed cardinality are closely related. 
The goal of this paper is twofold. First, we present a unified way of formulating numerical integration problems from both approximation theory and discrepancy theory. We present it in Sections \ref{I} and \ref{disc}. Mostly, these two sections contain 
known results and they can be considered as a survey. Second, we show how techniques, developed in approximation theory, work in proving lower bounds for 
the $r$-smooth $L_\bp$-discrepancy. These new results are presented in Sections
 \ref{L2} and \ref{Lp}. In Section \ref{D} we briefly discuss the lower bounds for different 
 types of discrepancy. 
 The main strategic point of Sections \ref{L2}, \ref{Lp}, and \ref{D} is to motivate the study of discrepancy for the whole range of smoothness from $r=1$, which corresponds to the classical setting, to arbitrarily large $r\in \N$. In Section \ref{nos}
 we move in other direction -- from classes of smoothness one to classes, for which we do not impose any smoothness assumptions. Following known results from \cite{VT89} and \cite{VT149}, we establish in Section \ref{nos} that even in such a general setting with no smoothness assumptions we can guarantee some rate of decay of errors of numerical integration.
 We hope that this paper will encourage researchers, working in 
 the discrepancy theory, to thoroughly study smooth discrepancy along with classical 
 discrepancy. 

We formulate the numerical integration problem in a general setting, which includes various discrepancy settings.  Numerical integration seeks good ways of approximating an integral
$$
\int_\Omega f(\bx)d\mu
$$
by an expression of the form
\be\label{1.1}
\La_m(f,\xi) :=\sum_{j=1}^m\la_jf(\xi^j),\quad \xi=(\xi^1,\dots,\xi^m),\quad \xi^j \in \Omega,\quad j=1,\dots,m. 
\ee
It is clear that we must assume that $f$ is integrable and defined at the points
 $\xi^1,\dots,\xi^m$. Expression (\ref{1.1}) is called a {\it cubature formula} $(\xi,\La)$ (if $\Omega \subset \R^d$, $d\ge 2$) or a {\it quadrature formula} $(\xi,\La)$ (if $\Omega \subset \R$) with knots $\xi =(\xi^1,\dots,\xi^m)$ and weights $\La:=(\la_1,\dots,\la_m)$. 
 
 Some classes of cubature formulas are of special interest. For instance, the Quasi-Monte Carlo cubature formulas, which have equal weights $1/m$, are important in applications. We use a special notation for these cubature formulas
 $$
 Q_m(f,\xi) :=\frac{1}{m}\sum_{j=1}^mf(\xi^j).
 $$
 
 The following class is a natural subclass of all cubature formulas. Let $B$ be a positive number and $Q(B,m)$ be the set of cubature formulas $\Lambda_m(\cdot,\xi)$ satisfying the additional condition
\be\label{1.2}
\sum_{\mu=1}^m |\lambda_\mu| \le B.
\ee
 
 For a function class $\bW$ we introduce a concept of error of the cubature formula $\La_m(\cdot,\xi)$ by
\be\label{1.3}
\La_m(\bW,\xi):= \sup_{f\in \bW} |\int_\Omega fd\mu -\La_m(f,\xi)|. 
\ee
The quantity $\La_m(\bW,\xi)$ is a classical characteristic of the quality of a given cubature formula $\La_m(\cdot,\xi)$. This setting is called {\it the worst case setting} in 
the Information Based Complexity. If the class $\bW=\{f(\bx,\by): \by \in Y\}$ is parametrized by a parameter $\by \in Y\subset \R^n$ with $Y$ being a bounded measurable set, then we can consider a natural {\it average case setting}. For $\bp =(p_1,\dots,p_n)$ define
\be\label{1.4}
\La_m(\bW,\xi,\bp):= \|\int_\Omega f(\cdot,\by)d\mu -\La_m(f(\cdot,\by),\xi)\|_\bp,
\ee
where the vector $L_\bp$ norm is taken with respect to the Lebesgue measure on $Y$. 
We write $\La_m(\bW,\xi,\infty):= \La_m(\bW,\xi)$. We are interested in dependence on $m$ of the quantities
$$
\kappa_m (\bW,\bp) := \inf_{ \lambda_1,\dots,\lambda_m; \xi^{1},\dots,
\xi^m}\Lambda_m(\bW,\xi,\bp)
$$
for different classes $\bW$. We begin with a rather general setting and consider 
particular examples later.  Let $1\le q\le \infty$. We define a set $\mathcal K_q$ of kernels possessing the following properties. 
 Let $K(\bx,\by)$ be a measurable function on $\Omega^1\times\Omega^2$.
 We assume that for any $\bx\in\Omega^1$ we have $K(\bx,\cdot)\in L_q(\Omega^2)$; for any $\by\in \Omega^2$ the $K(\cdot,\by)$ is integrable over $\Omega^1$ and $\int_{\Omega^1} K(\bx,\cdot)d\bx \in L_q(\Omega^2)$. For $1\le p\le \infty$ and a kernel $K\in \mathcal K_{p'}$, $p':=p/(p-1)$, we define the class
\be\label{1.5}
\bW^K_p :=\{f:f=\int_{\Omega^2}K(\bx,\by)\varphi(\by)d\by,\quad\|\varphi\|_{L_p(\Omega^2)}\le 1\}.  
\ee
Then each $f\in \bW^K_p$ is integrable on $\Omega^1$ (by Fubini's theorem) and defined at each point of $\Omega^1$. We denote for convenience
$$
 J_K(\by):=\int_{\Omega^1}K(\bx,\by)d\bx.
$$

For a cubature formula $\Lambda_m(\cdot,\xi)$ we have
$$
\Lambda_m(\bW^K_p,\xi) = \sup_{\|\varphi\|_{L_p(\Omega^2)}\le 1} |\int_{\Omega^2}\bigl( J_K(\by)-\sum_{\mu=1}^m\lambda_\mu K(\xi^\mu,\by)\bigr)\varphi(\by)d\by|=
$$
\be\label{1.6}
=\| J_K(\cdot)-\sum_{\mu=1}^m\lambda_\mu K(\xi^\mu,\cdot)\|_{L_{p'}(\Omega^2)}.
\ee
Consider a problem of numerical integration of functions  $K(\bx,\by)$, $\by\in\Omega^2$, with respect to  $\bx$,  $K\in {\mathcal K}_q$, in other words functions from the function class $\bK:=\{K(\bx,\by):\by\in\Omega^2\}$:
$$
{ \int_{\Omega^1} K(\bx,\by)d\bx - \sum_{\mu=1}^m \lambda_\mu K(\xi^\mu,\by)}.
$$

\begin{Definition}\label{D1.1}  $(K,q)$-discrepancy of a set of 
knots  $\xi^1,\dots,\xi^m$ and a set of weights  $\lambda_1,\dots,\lambda_\mu$ (a cubature formula  $(\xi,\Lambda)$) is
$$
 D(\xi,\Lambda,K,q):=\Lambda_m(\bK,\xi,q)=\|\int_{\Omega^1} K(\bx,\by)d\bx - \sum_{\mu=1}^m \lambda_\mu K(\xi^\mu,\by)\|_{L_q(\Omega^2)}.
$$
\end{Definition} 
In a special case $\La_m(\cdot,\xi) = Q_m(\cdot,\xi)$ we write $D(\xi,Q,K,q)$.
The above definition of the $(K,q)$-discrepancy and relation (\ref{1.6})  imply right a way the following relation
\be\label{1.7}
D(\xi,\Lambda,K,p') = \Lambda_m(\bW^K_p,\xi).
\ee
Relation (\ref{1.7}) shows that numerical integration in the class $\bW^K_p$ and the $(K,q)$-discrepancy are 
tied by the duality principle.

Let us consider a special case, when $K(\bx,\by)=F(\bx-\by)$, $\Omega^1=\Omega^2 = [0,1)^d$ and we deal with $1$-periodic in each variable functions. 
Associate with a cubature formula $(\xi,\La)$ and the function $F$ the following function
$$
g_{\xi,\Lambda,F}( \bx) := \sum_{ \bk}\Lambda(\xi,\bk)\hat F
( \bk)e^{2\pi i( \bk, \bx)}- \hat F(\mathbf 0),
$$
where
$$
\Lambda(\xi,\bk):= \La_m(e^{2\pi i(\bk,\bx)},\xi).
$$
Then for the quantity $\Lambda_m ( \bW_{p}^F ,\xi)$
we have ($p' := p/(p-1)$)
$$
\Lambda_m ( \bW_{p}^F ,\xi)=
\sup_{f\in \bW_{p}^F}
\bigl| \Lambda_m(f,\xi) -\hat f(\mathbf 0)\bigr|
$$
$$
=\sup_{\|\varphi\|_p\le 1} \bigl| \Lambda_m \bigl(F( \bx)\ast
\varphi( \bx),\xi\bigr) -\hat F(\mathbf 0)\hat\varphi(\mathbf 0)\bigr| 
$$
\be\label{1.8}
=\sup_{\|\varphi\|_p\le 1}\bigl|\<g_{\xi,\Lambda,F}(-\by),
\overline{\varphi(\by)}\>\bigr|=\|g_{\xi,\Lambda,F}\|_{p'}.
\ee
Let us discuss a special case of function $F$, which is very important in numerical
integration (see, for instance, \cite{TBook}, \cite{VT89}, and \cite{DTU}).  Let for $r>0$ 
 \be\label{1.9}
F_{r,\alpha}(x):= 1+2\sum_{k=1}^\infty k^{-r}\cos (2\pi kx-\alpha \pi/2).
\ee
For $\bx=(x_1,\dots,x_d)$, $\alpha=(\alpha_1,\dots,\alpha_d)$ denote
$$
F_{r,\alpha}(\bx) := \prod_{j=1}^d F_{r,\alpha_j}(x_j)
$$
and
$$
\bW^r_{p,\alpha} :=\bW^{F_{r,\alpha}}_p= \{f:f(\bx)=(F_{r,\alpha}\ast \varphi)(\bx)
$$
$$
:= \int_{[0,1)^d} F_{r,\alpha}(\bx-\by)\varphi(\by)d\by,\quad \|\varphi\|_p \le 1\}.
$$ 
  In the case of integer $r$ the class $\bW^r_{p,\alpha}$ with $\alpha=(r,\dots,r)$ is very close to the class of functions $f$, 
satisfying $\|f^{(r,\dots,r)}\|_p \le 1$, where $f^{(r,\dots,r)}$ is the mixed derivative of $f$ of order $rd$. 

It is easy to see that
\be\label{1.10}
\|g_{\xi,\Lambda,F_{r,\alpha}}\|_{2} = \left(\sum_{\bk\neq \mathbf 0}|\Lambda(\xi,\bk)|^2\left(\prod_{j=1}^d (\max(|k_j|,1))^{-r}\right)^2 + |\La(\xi,\mathbf 0)-1|^2\right)^{1/2} .
\ee
The above quantity in the case $r=1$ was introduced in \cite{Zi} under the name 
{\it diaphony}. In case of generic $r$ it was called {\it generalized diaphony} and was studied in \cite{Lev7}. Relation (\ref{1.8}) shows that generalized diaphony is closely related to numerical integration of the class $\bW^r_{2,\alpha}$. Following this analogy,
we can call the quantity $\|g_{\xi,\Lambda,F_{r,\alpha}}\|_{q}$ the $(r,q)$-diaphony of the pair $(\xi,\La)$ (the cubature formula $(\xi,\La)$).

\section{Discrepancy}
\label{disc}

We now describe some typical classes $\bW$, which are of interest in numerical integration and in discrepancy theory. We begin with a classical definition of discrepancy ("star discrepancy", $L_\infty$-discrepancy) of a point set $\xi := \{\xi^\mu\}_{\mu=1}^m\subset [0,1)^d$. 
Let $d\ge 2$ and $[0,1)^d$ be the $d$-dimensional unit cube. For convenience we sometimes use the notation $\Omega_d:=[0,1)^d$. For $\bx,\by \in [0,1)^d$ with $\bx=(x_1,\dots,x_d)$ and $\by=(y_1,\dots,y_d)$ we write $\bx < \by$ if this inequality holds coordinate-wise. For $\bx<\by$ we write $[\bx,\by)$ for the axis-parallel box $[x_1,y_1)\times\cdots\times[x_d,y_d)$ and define
$$
\cB:= \{[\bx,\by): \bx,\by\in [0,1)^d, \bx<\by\}.
$$

Introduce a class of special $d$-variate characteristic functions
$$
\chi^d := \{\chi_{[\mathbf 0,\bb)}(\bx):=\prod_{j=1}^d \chi_{[0,b_j)}(x_j),\quad b_j\in [0,1),\quad j=1,\dots,d\}
$$
where $\chi_{[a,b)}(x)$ is a univariate characteristic function of the interval $[a,b)$. 
The classical definition of discrepancy of a set $\xi$ of points $\{\xi^1,\dots,\xi^m\}\subset [0,1)^d$ is as follows
$$
D_\infty(\xi)  :=Q_m(\chi^d,\xi,\infty)= \max_{\bb\in [0,1)^d}\left|\prod_{j=1}^db_j -\frac{1}{m}\sum_{\mu=1}^m \chi_{[\mathbf 0,\bb)}(\xi^\mu)\right|.  
$$
The class $\chi^d$ is parametrized by the parameter $\bb\in [0,1)^d$. Therefore, we can define the $L_q$-discrepancy, $1\le q\le\infty$, of $\xi$ as follows
$$
D_q(\xi) :=   Q_m(\chi^d,\xi,q).
$$
In the above definitions the function class consists of characteristic functions, which have smoothness $1$ in the $L_1$ norm. In numerical integration it is natural to study 
function classes with arbitrary smoothness $r$. There are different generalizations of 
the above concept of discrepancy to the case of {\it smooth  discrepancy}. We discuss two of them here. In the definition of the first version of the $r$-discrepancy (see \cite{TBook}) instead of the characteristic function (this corresponds to $1$-discrepancy) we use the following function
\begin{align*}
B_r(\bx,\by)&:= \prod_{j=1}^d\bigl((r-1)!\bigr)^{-1}
(y_j - x_j )_+^{r-1},\\
  \bx,\by&\in\Omega_d,\qquad (a)_+ := \max (a,0).
\end{align*}
Denote
$$
\bB^{r,d}:=\{B_r(\bx,\by): \by \in \Omega_d\}.
$$
Then for a point set $\xi:=\{\xi^\mu\}_{\mu=1}^m$ of cardinality $m$ and weights $\Lambda:=\{\la_\mu\}_{\mu=1}^m$ we define the $r$-discrepancy of the pair $(\xi,\Lambda)$ by the formula
$$
D^r_q (\xi,\Lambda):=D(\xi,\Lambda,B_r,q)=\La_m(\bB^{r,d},\xi,q)
$$
\be\label{2.1}
 = \left \|\sum_{\mu=1}^{m}\lambda_{\mu}B_r (\xi^{\mu},\by)-
\prod_{j=1}^d (y_j^r /r!)\right\|_q  .
\ee

Consider the class
$ \dot{ \bW}_{p}^r  :=\bW^{B_r}_p$  consisting of the functions
$f(\bx)$ representable in the form
$$
f(\bx) =\int_{\Omega_d} B_r (\bx,\by)\varphi(\by) d \by,
\qquad \|\varphi\|_p \le 1.
$$
In connection with the definition of the class
$ \dot{ \bW}_{p}^r $ we remark here that
for the error of the cubature formula $(\xi,\Lambda)$ with
weights $\Lambda = (\lambda_1,\dots,\lambda_m)$ and knots
$\xi = (\xi^1,\dots,\xi^m)$ the following relation holds with  $p' := p/(p-1)$ (see (\ref{1.7}))
\be\label{2.2}
\Lambda_m\bigl( \dot{ \bW}_{p}^r ,\xi\bigr)=\left \|\sum_{\mu=1}^{m}\lambda_{\mu}B_r ( \btt,\xi^{\mu})-
\prod_{j=1}^d (t_j^r /r!)\right\|_{p'} =
D^r_{p'} (\xi,\Lambda)  .
\ee
Thus, errors of numerical integration of classes $\dot{ \bW}_{p}^r$ are dual to the average errors of numerical integration of classes $\bB^{r,d}$. 

Note that discrepancy $D_{\infty}(\xi) $ is equivalent within multiplicative constants, which may only depend on $d$, to the following quantity
\be\label{2.3}
D^1_\infty(\xi):=  \sup_{B\in\cB}\left|vol(B)-\frac{1}{m}\sum_{\mu=1}^m \chi_B(\xi^\mu)\right|,
\ee
where for $B=[\ba,\bb)\in \cB$ we denote $\chi_B(\bx):= \prod_{j=1}^d \chi_{[a_j,b_j)}(x_j)$. In particular, this implies that $D^1_{\infty} (\xi)$ is equivalent to its 
periodic analog, which we define below. For a function $f\in L_1(\R^d)$ with a compact support we define its periodization $\tilde f$ as follows
$$
\tilde f(\bx) := \sum_{\bm\in \Z^d} f(\bm+\bx).
$$
Define
$$
\tilde \chi^d :=\{ \tilde \chi_{[\ba,\ba+\bb)}(\bx): \ba,\bb \in \Omega_d\}.
$$
The class $\tilde \chi^d$ is parametrized by $\by =(\ba,\bb)\in \Omega_d\times\Omega_d \subset \R^{2d}$. Let $\bp=(p_1,\dots,p_1,p_2,\dots,p_2)$ 
with first $d$ coordinates $p_1$ and the rest $p_2$. Define
$$
\tilde D^1_\bp(\xi) :=   Q_m(\tilde \chi^d,\xi,\bp).
$$
Let us make some remarks about the above defined concepts of discrepancy. 
In the definition of discrepancy $D_p(\xi)$ the vertex $\mathbf 0$ plays a dominating role compared to other vertexes of the unit cube. For this reason discrepancy $D_p(\xi)$ is 
called {\it $L_p$-discrepancy for corners} (see \cite{Mat}) or {\it anchored discrepancy}. 
The definition of $D^1_\infty(\xi)$ and $\tilde D^1_\bp(\xi)$ are more symmetric with respect to vertexes of the unit cube. As we pointed out above in the case $\bp=\infty$ 
the quantities $D_\infty(\xi)$, $D^1_\infty(\xi)$ and $\tilde D^1_\infty(\xi)$ are equivalent.  However, as it is shown in \cite{Lev2} for $p=2$, the quantities $D_p(\xi)$ and $\tilde D^1_p(\xi)$ may behave differently (see \cite{Lev2} for detailed comparison of $D_2(\xi)$ and $\tilde D^1_2(\xi)$). In the definition of class $\tilde \chi^d$ parameters $\ba$ and $\bb$ play different roles. Parameter $\bb$ controls the shape of the support of the corresponding characteristic function. This parameter has the same role in the definition of $D_p(\xi)$. Parameter $\ba$ controls the shift of the characteristic function. In the case of function class $\chi^d$ there is no shift, $\ba=\mathbf 0$, and all characteristic functions are anchored at $\mathbf 0$. Thus, averaging over 
parameter $\ba$ -- the shift -- is a new feature of discrepancy $\tilde D^1_\bp(\xi)$. V. Lev (see, for instance, \cite{Lev7}), arguing that "it is this kind of discrepancy which was originally considered in the pioneering paper of Weyl" (see \cite{W}), suggests to call this type of discrepancy {\it Weyl discrepancy}. 

We now proceed to the $r$-smooth discrepancy. It is more convenient for us to consider the average setting in the periodic case. 
For $r=1,2,3,\dots$ we inductively define
$$
h^1(x,u):= \chi_{[-u/2,u/2)}(x), 
$$
$$
h^r(x,u) := h^{r-1}(x,u)\ast h^1(x,u),\qquad r=2,3,4,\dots,
$$
where
$$
f(x)\ast g(x) := \int_\R f(x-y)g(y)dy.
$$
Then $h^r(x,u)$ has smoothness $r$ in $L_1$ and has support $(-ru/2,ru/2)$. 
For a box $B$ represented in the form
$$
B= \prod_{j=1}^d [z_j-ru_j/2,z_j+ru/2)
$$
 define
$$
h^r_B(\bx):= h^r(\bx,\bz,\bu):=\prod_{j=1}^d h^r(x_j-z_j,u_j).
$$
 
Consider $\bu \in (0,\frac{1}{2}]^d$. Then for all $\bz\in [0,1)^d$ we have 
$$
\supp(h^r(\bx,\bz,\bu)) \subset (-r/4,1+r/4)^d.
$$
Now, for each $\bz\in [0,1)^d$ consider a periodization of function $h^r(\bx,\bz,\bu)$ 
in $\bx$ with period $1$ in each variable  
$\tilde h^r(\bx,\bz,\bu)$. Consider the class of periodic $r$-smooth hat functions
$$
\bH^{r,d}:=\{\tilde h^r(\bx,\bz,\bu): \bz\in[0,1)^d;\bu\in(0,1/2]^d\}.
$$

Define the corresponding {\it periodic $r$-smooth discrepancy}   as follows
$$
\tilde D^{r}_\infty(\xi,\La):=  \La_m(\bH^{r,d},\xi)
$$
\be\label{2.4}
   = \sup_{\bz\in[0,1)^d;\bu\in(0,1/2]^d}\left|\int_{[0,1)^d} \tilde h^r(\bx,\bz,\bu)d\bx- \sum_{\mu=1}^m \la_\mu \tilde h^r(\xi^\mu,\bz,\bu)\right|.
\ee

For $1\le p_1,p_2\le \infty$,  define the corresponding {\it periodic $r$-smooth $L_\bp$-discrepancy} (Weyl $r$-smooth $L_\bp$-discrepancy) as follows (see \cite{VT165} for the case $p=\infty$)
\be\label{2.5}
  \tilde D^{r}_{p_1,p_2}(\xi,\La):=
  \left\|\|\int_{[0,1)^d} \tilde h^r(\bx,\bz,\bu)d\bx- \sum_{\mu=1}^m \la_\mu \tilde h^r(\xi^\mu,\bz,\bu)\|_{p_1}\right\|_{p_2}
\ee
where the $L_{p_1}$ norm is taken with respect to $\bz$ over the unit cube $[0,1)^d$ and the $L_{p_2}$ norm is taken with respect to $\bu$ over the cube $(0,1/2]^d$. In the definition of $\tilde D^{r}_{p_1,p_2}(\xi,\La)$ parameters $\bz$ and $\bu$ play different roles. The most important parameter is $\bu$ -- it controls the shape of supports of the corresponding hat functions. It seems like the most natural value for parameter $p_2$ is $\infty$. In this case we obtain bounds uniform with respect to the shape and the size of supports of hat functions. Recently, it was noticed that the following subclasses of $\bH^{r,d}$ are of interest in studying dispersion (see \cite{VT163})
$$
\bH^{r,d}(v):=\{\tilde h^r(\bx,\bz,\bu): \bz\in[0,1)^d,\,\bu\in(0,1/2]^d,\, pr(\bu)=v\},
$$
where $pr(\bu):= \prod_{j=1}^d u_j$. We call the corresponding characteristics 
$$
\La_m(\bH^{r,d}(v),\xi,\infty), \qquad \La_m(\bH^{r,d}(v),\xi,(p,\infty))
$$
the {\it fixed volume $r$-smooth discrepancy} and {\it fixed volume $r$-smooth $L_p$-discrepancy} respectively.

\section{Lower bounds in case $p=2$}
\label{L2}

 \begin{Theorem}\label{T3.1} Let $r\in\N$. Then for any $(\xi,\La)$ we have
$$
    \tilde D^{r}_{2,2}(\xi,\La) \geq   C(r,d)   m^{-r}(\log   m)^{(d-1)/2},   \qquad
C(r,d)>0.
$$
\end{Theorem}
\begin{proof} We use a  notation
$$
     \La(\xi,\bk) := \Lambda_m(e^{i2\pi(\bk,\bx)},\xi) = \sum_{\mu=1}^m  \lambda_\mu
e^{i2\pi(\bk,\xi^\mu)}.
$$
We need a known result on the lower bound for the weighted sum of 
$\{|\La(\xi,\bk)|^2\}$ (see \cite{TBook} and \cite{VT89}).

\begin{Lemma}\label{L3.1} The following inequality is valid for any
$r > 1$
$$
\sum_{ \bk\ne 0}\bigl|\Lambda( \xi,\bk)\bigr|^2 pr
(\bar{ \bk})^{-r}\ge
C(r,d)\bigl|\Lambda(\xi,\mathbf 0)\bigr|^2 m^{-r}(\log m)^{d-1},
$$
where $\bar{k_j}:=\max(|k_j|,1)$ and $pr(\bar{\bk}):=\prod_{j=1}^d \bar{k_j}$.
\end{Lemma}

We have
$$
\delta(\bz,\bu):=  \sum_{\mu=1}^m \la_\mu \tilde h^r(\xi^\mu,\bz,\bu) -\int_{[0,1)^d} \tilde h^r(\bx,\bz,\bu)d\bx
$$
$$
=(\Lambda(\xi,\mathbf 0)-1)\hat{\tilde h}^r(\mathbf 0,\bz,\bu)+ \sum_{\bk\neq\mathbf 0} \La(\xi,\bk) \hat{\tilde h}^r(\bk,\bz,\bu).
$$
Taking into account the formula for a $1$-periodic $f(\bx)$
$$
\int_{[0,1)^d} f(\bx-\bz)e^{-2\pi i (\bk,\bx)} d\bx = e^{-2\pi i (\bk,\bz)}\int_{[0,1)^d} f(\bx)e^{-2\pi i (\bk,\bx)} d\bx,
$$
we obtain
\be\label{3.0}
\delta(\bz,\bu)=(\Lambda(\xi,\mathbf 0)-1)\hat{\tilde h}^r(\mathbf 0,\mathbf 0,\bu)+\sum_{\bk\neq\mathbf 0} \La(\xi,\bk) \hat{\tilde h}^r(\bk,\mathbf 0,\bu)e^{-2\pi i (\bk,\bz)}.
\ee
Therefore,
\be\label{3.1}
\|\delta(\cdot,\bu)\|_2^2 \ge \sum_{\bk\neq\mathbf 0} |\La(\xi,\bk) \hat{\tilde h}^r(\bk,\mathbf 0,\bu)|^2.
\ee
Next
\be\label{3.1'}
\hat {\tilde h}^r(0, 0, u)=u^r,\quad \hat {\tilde h}^r(k, 0, u)=\left(\frac{\sin(\pi ku)}{\pi k}\right)^r,\quad k\neq 0,
\ee
which implies  
$$
\int_0^{1/2}|\hat {\tilde h}^r(k, 0, u)|^2du \ge c(r)(\bar k)^{-2r}.
$$
Integrating the right hand side of (\ref{3.1}) with respect to $\bu$ over $(0,1/2]^d$ and using Lemma \ref{L3.1}  we get  
\be\label{3.2}
\int_{(0,1/2]^d}\|\delta(\cdot,\bu)\|_2^2d\bu \ge C(r,d)|\Lambda(\xi,\mathbf 0)|^2 m^{-2r} (\log m)^{d-1}. 
\ee
We have 
$$
|\int_{[0,1)^d} \delta(\bz,\bu)d\bz| = |(1-\Lambda(\xi,\mathbf 0))\hat{\tilde h}^r(\mathbf 0,\mathbf 0,\bu)|.
$$
Therefore, it is clear that it must be $|\Lambda(\xi,\mathbf 0)|\ge c(r,d)>0$. 
This combined with (\ref{3.2}) completes the proof of Theorem \ref{T3.1}.

\end{proof}

We now prove that Theorem \ref{T3.1} is sharp. 
\begin{Proposition}\label{P3.1} For $r\in \N$ there exists a cubature formula $(\xi,\La)$
such that 
$$
    \tilde D^{r}_{2,\infty}(\xi,\La) \le   C(r,d)   m^{-r}(\log   m)^{(d-1)/2},   \qquad
C(r,d)>0.
$$
\end{Proposition}
\begin{proof} We use the notation from the above proof of Theorem \ref{T3.1}. It follows from (\ref{3.0})   that for each $\bu$
\be\label{3.5}
\|\delta(\cdot,\bu)\|_2^2 = |(\Lambda(\xi,\mathbf 0)-1)\hat{\tilde h}^r(\mathbf 0,\mathbf 0,\bu)|^2+\sum_{\bk\neq\mathbf 0} |\La(\xi,\bk) \hat{\tilde h}^r(\bk,\mathbf 0,\bu)|^2.
\ee
This and (\ref{3.1'}) imply that for $\bu\in (0,1/2]^d$ (see (\ref{1.10}))
\be\label{3.6}
\|\delta(\cdot,\bu)\|_2 \le \|g_{\xi,\Lambda,F_{r,0}}\|_{2}.
\ee
By (\ref{1.8}) we obtain
\be\label{3.7}
\|g_{\xi,\Lambda,F_{r,0}}\|_{2} = \La_m(\bW^r_{2,0},\xi).
\ee
It is known (see \cite{Fro1}, \cite{By}, \cite{TBook}, and \cite{VT89}) that there exists a cubature formula (actually independent of $r$) based on the Frolov lattice such that
\be\label{3.8}
 \La_m(\bW^r_{2,0},\xi) \le C(r,d) m^{-r}(\log   m)^{(d-1)/2}.
 \ee
 Combining (\ref{3.6})--(\ref{3.8}) we complete the proof.

\end{proof}

We now show how Lemma \ref{L3.1} can be used to obtain a result similar to Theorem \ref{T3.1} for cubes instead of boxes of arbitrary shape. This result is in a style of 
results by Beck and Montgomery (their results correspond to the case $r=1$, see \cite{BC}, p. 132). 
Denote
$$
{\tilde h}^r(\bx,\mathbf z,u) :={\tilde h}^r(\bx,\mathbf z,(u,\dots,u)).
$$
Consider the following subclass of the class $\bH^{r,d}$
$$
\bH^{r,d,c}:= \{\tilde h^r(\bx,\bz,u): \bz\in [0,1)^d,\,   u\in (0,1/2]\}.
$$
Then the class $\bH^{r,d,c}$ is parametrized by $\bz\in [0,1)^d$ and $u\in (0,1/2]$. 
For $1\le p_1,p_2\le \infty$,  define the corresponding {\it periodic $r$-smooth $L_\bp$-discrepancy for cubes} (Weyl $r$-smooth $L_\bp$-discrepancy for cubes) as follows  
\be\label{3.9}
  \tilde D^{r,c}_{p_1,p_2}(\xi,\La):=
  \left\|\|\int_{[0,1)^d} \tilde h^r(\bx,\bz,u)d\bx- \sum_{\mu=1}^m \la_\mu \tilde h^r(\xi^\mu,\bz,u)\|_{p_1}\right\|_{p_2}
\ee
where the $L_{p_1}$ norm is taken with respect to $\bz$ over the unit cube $[0,1)^d$ and the $L_{p_2}$ norm is taken with respect to $u$ over the interval $(0,1/2]$. 

 \begin{Theorem}\label{T3.2} Let $r\in\N$. Then for any $(\xi,\La)$ we have
$$
    \tilde D^{r,c}_{2,2}(\xi,\La) \geq   C(r,d)   m^{-r}(\log   m)^{(d-1)/2},   \qquad
C(r,d)>0.
$$
\end{Theorem}
\begin{proof} We obtain from (\ref{3.5})
\be\label{3.10}
\|\delta(\cdot,u)\|_2^2=|(\Lambda(\xi,\mathbf 0)-1)\hat{\tilde h}^r(\mathbf 0,\mathbf 0,u)|^2+\sum_{\bk\neq\mathbf 0} |\La(\xi,\bk) \hat{\tilde h}^r(\bk,\mathbf 0,u)|^2 .
\ee
Note that for $\eta\in (0,1/\pi)$ and $k\neq 0$
$$
|\{u\in (0,1/2]:\, |\sin(\pi ku)| \le \sin(\pi \eta)\}| \le C\eta.
$$
Then, using (\ref{3.1'}), we obtain 
\be\label{3.11}
\int_{(0,1/2]}|\hat{\tilde h}^r(\bk,\mathbf 0,u)|^2du \ge C(r,d)pr(\bar\bk)^{-2r},\quad C(r,d)>0.
\ee
Therefore, Lemma \ref{L3.1} and (\ref{3.10}) imply the required inequality.
\end{proof}

Theorem \ref{T3.2} implies that for any cubature formula $(\xi,\La)$ there exists 
$\bz\in [0,1)^d$ and $u\in (0,1/2]$ such that the error of numerical integration of 
the function ${\tilde h}^r(\bx,\mathbf z,u)$ by the cubature formula $(\xi,\La)$ is 
greater than $C(r,d)   m^{-r}(\log   m)^{(d-1)/2}$. We now demonstrate that the anchored setting (when $\bz=\mathbf 0$) gives different results. For illustration purposes we 
only consider the simplest case $r=1$, when we numerically integrate the characteristics functions of cubes $\chi_{[\mathbf 0,\by)}(\bx)$ with $\by = (y,\dots,y)$, $y\in [0,1)$. Consider the function  
$$
w=f(y):= \int_{[0,1)^d} \chi_{[\mathbf 0,(y,\dots,y))}(\bx)d\bx = y^d,\quad y\in [0,1).
$$
For $k=1,\dots,m+1$ define
$$
w_k := \frac{k}{m+1},\quad y_k := (w_k)^{1/d}=\left(\frac{k}{m+1}\right)^{1/d}.
$$
Define the cubature formula $(\xi,\La)$ in the following way
$$
\xi^k:= (y_k,0,\dots,0),\quad \la_k := \frac{1}{m+1},\quad k=1,\dots,m.
$$
Let $y\in [y_n,y_{n+1})$ with some $1\le n \le m$. Then 
$$
|\La_m(\chi_{[\mathbf 0,(y,\dots,y))},\xi)- f(y)| \le w_{n+1}-w_n = \frac{1}{m+1}.
$$
This shows that the optimal error of numerical integration of characteristic functions of anchored cubes is better than the one for shifted cubes.

\section{A lower bound in case $1<p<2$}
\label{Lp}

We remind some known results, which we use in this section. 
 The following theorem is proved in \cite{VT43} (see also \cite{TBook} and \cite{VT89}).
\begin{Theorem}\label{T4.1} The following lower estimate is valid for any
cubature formula $(\xi,\Lambda)$ with $m$ knots $(r > 1/p)$
$$
\Lambda_m( \bW_{p,\alpha}^r,\xi) \ge C(r,d,p)m^{-r}
(\log m)^{\frac{d-1}{2}},\qquad 1 \le p < \infty .
$$
\end{Theorem}

Thus, (\ref{1.8}) and Theorem \ref{T4.1} imply the following lemma.

\begin{Lemma}\label{L4.1} Let $1<p<\infty$ and $r>1-1/p$. Then for any $(\xi,\Lambda)$ we have
$$
\|g_{\xi,\Lambda,F_{r,\mathbf 0}}\|_{p} \ge C(r,d,p)m^{-r}
(\log m)^{\frac{d-1}{2}}.
$$
\end{Lemma}
Note, that it follows from the definition of $F_{r,\alpha}$ that
$\hat F_{r,\mathbf 0}(\bk) = \prod_{j=1}^d(\bar k_j)^{-r}$. Let $c$ be a positive constant. Denote $k^c$ to be 
$c$ if $k=0$ and $k^c=k$ otherwise, $\bk^c:=(k_1^c,\dots,k_d^c)$. It is not difficult to see that Lemma \ref{L4.1} implies the following lemma.

\begin{Lemma}\label{L4.2}   Define
$$
g_{\xi,\Lambda,r}^c( \bx) := \sum_{ \bk\neq \mathbf 0}\Lambda(\xi, \bk)\left(\prod_{j=1}^d
( k_j^c)^{-r}\right)e^{2\pi i( \bk, \bx)}+c^{-rd}(\Lambda(\xi,\mathbf 0)- 1).
$$
Let $1<p<\infty$ and $r>1-1/p$. Then for any $(\xi,\Lambda)$ we have 
$$
\|g_{\xi,\Lambda,r}^c\|_{p} \ge C(r,d,p,c)m^{-r}
(\log m)^{\frac{d-1}{2}} .
$$
\end{Lemma}

We now prove an analog of Theorem \ref{T3.1} for $p>1$. 

 \begin{Theorem}\label{T4.2} Let $r\in\N$ be an even number. Then for any $(\xi,\La)$ we have for $1<p<\infty$
$$
    \tilde D^{r}_{p,1}(\xi,\La) \geq   C(r,d,p)   m^{-r}(\log   m)^{(d-1)/2},   \qquad
C(r,d,p)>0.
$$
\end{Theorem}
\begin{proof} As above we use the error function (see (\ref{3.0}))
$$
\delta(\bz,\bu)=(\Lambda(\xi,\mathbf 0)-1)\hat{\tilde h}^r(\mathbf 0,\mathbf 0,\bu)+\sum_{\bk\neq\mathbf 0} \La(\xi,\bk) \hat{\tilde h}^r(\bk,\mathbf 0,\bu)e^{-2\pi i (\bk,\bz)}.
$$
Consider a function
$$
f(\bz):= \int_{(0,1/2]^d}\delta(\bz,\bu)d\bu =(\Lambda(\xi,\mathbf 0)-1)\int_{(0,1/2]^d}\hat{\tilde h}^r(\mathbf 0,\mathbf 0,\bu)d\bu
$$
\be\label{4.1}
    +\sum_{\bk\neq\mathbf 0} \La(\xi,\bk) e^{-2\pi i (\bk,\bz)}\int_{(0,1/2]^d}\hat{\tilde h}^r(\bk,\mathbf 0,\bu)d\bu.
\ee
Next
\be\label{4.1'}
\hat {\tilde h}^r(0, 0, u)=u^r,\quad \hat {\tilde h}^r(k, 0, u)=\left(\frac{\sin(\pi ku)}{\pi k}\right)^r,\quad k\neq 0,
\ee
which implies  for $k\neq 0$ and even $r$
$$
\int_0^{1/2}\hat {\tilde h}^r(k, 0, u)du = (\pi k)^{-r}\int_0^{1/2} ((1-\cos 2\pi ku)/2)^{r/2}du = (\pi k)^{-r} c(r).
$$
Thus, 
$$
f(\bz) = (\pi^{-r}c(r))^d g_{\xi,\Lambda,r}^c(-\bz)
$$
with some $c=c'(r)$. Therefore, by Lemma \ref{L4.2} we get
\be\label{4.2}
\|f\|_p \ge C(r,d,p)m^{-r}(\log m)^{\frac{d-1}{2}}. 
\ee
Further,
$$
\tilde D^{r}_{p,1}(\xi,\La)=\int_{(0,1/2]^d}\|\delta(\cdot,\bu)\|_pd\bu
\ge  \left\|\int_{(0,1/2]^d} \delta(\cdot,\bu)d\bu\right\|_p =  \|f\|_p.
$$
It remains to use (\ref{4.2}).
\end{proof}

 We now prove that the lower bound in Theorem \ref{T4.2} is sharp. The following 
 result is an extension of Proposition \ref{P3.1}. 
 \begin{Proposition}\label{P4.1} For $r\in \N$ and $1<p<\infty$ there exists a cubature formula $(\xi,\La)$
such that 
$$
    \tilde D^{r}_{p,\infty}(\xi,\La) \le   C(r,p,d)   m^{-r}(\log   m)^{(d-1)/2}.
$$
\end{Proposition}
\begin{proof} Proposition \ref{P4.1} follows from Lemma \ref{L4.3} below and known 
results on $\kappa_m(\bW^r_{q,\alpha},\infty)$. 
\begin{Lemma}\label{L4.3} Let $1<p<\infty$ and $r\in \N$. For any cubature formula $(\xi,\La)$ we have
$$
\tilde D^r_{p,\infty}(\xi,\La) \le C(p,d)\La_m(\bW^r_{p',\alpha},\xi).
$$
\end{Lemma} 
\begin{proof} It is known that for $1<q<\infty$ classes $\bW^r_{q,\alpha}$ are equivalent for all $\alpha$. Therefore, it is sufficient to prove Lemma \ref{L4.3} in the case $\alpha =\mathbf 0$. By (\ref{1.8}) 
\be\label{4.3}
  \La_m(\bW^r_{p',\mathbf 0},\xi) = \|g_{\Lambda,\xi,F_{r,\mathbf 0}}\|_{p} .
\ee
Consider two univariate multiplier operators
$$
T(f):= \sum_{k\ge 0} \hat f (k) e^{2\pi i kx}  .
$$
$$
M_u(f) := \hat f(0) + \sum_{k\neq 0} \sin(\pi ku)  \hat f (k) e^{2\pi i kx}.
$$
It is known that operator $T$ is bounded as an operator from $L_q$ to $L_q$ for $1<q<\infty$ (it follows from the Riesz theorem). The observation
$$
\sum_{k} e^{i\pi ku}  \hat f (k) e^{2\pi i kx} = f(x+u/2)
$$
implies that
$$
\|M_u\|_{L_q\to L_q} \le 1.
$$
Denote by $T_j$ and $M_{u,j}$ operators $T$ and $M_u$ acting on a function $f(x_1,\dots,x_d)$ as a function on variable $x_j$. The above boundedness of operators $T$ and $M_u$ implies that 
$$
\|\prod_{j=1}^d T_j M_{u_j,j}^r\|_{L_q\to L_q} \le C(q,d).
$$
 Then for
$$
\delta(\bz,\bu)=(\Lambda(\xi,\mathbf 0)-1)\hat{\tilde h}^r(\mathbf 0,\mathbf 0,\bu)+\sum_{\bk\neq\mathbf 0} \La(\xi,\bk) \hat{\tilde h}^r(\bk,\mathbf 0,\bu)e^{-2\pi i (\bk,\bz)},
$$
taking into account (\ref{4.1'}), we obtain for each $\bu\in (0,1/2]^d$ 
$$
\|\delta(\cdot,\bu)\|_p \le C(p,d)\|g_{\Lambda,\xi,F_{r,\mathbf 0}}\|_{p}.
$$
This completes the proof of Lemma \ref{L4.3}
\end{proof}
We now complete the proof of Proposition \ref{P4.1}. It is known (see \cite{Fro1} for $q=2$ and \cite{Sk} for $1<q<\infty$) that for $r\in \N$ and $1<q<\infty$
$$
\kappa_m(\bW^r_{q,\alpha},\infty) \le C(r,q,d) m^{-r}(\log   m)^{(d-1)/2}.
$$
The above bound with $q=p'$ and Lemma \ref{L4.3} imply Proposition \ref{P4.1}.

\end{proof}

\section{Numerical integration without smoothness assumptions}
\label{nos}

In the previous Sections \ref{disc} -- \ref{Lp} we discussed numerical integration for 
classes of functions under certain conditions on smoothness. Parameter $r$ controlled 
the smoothness. The above results show that the numerical integration characteristics 
decay with the rate $m^{-r}(\log m)^{c(d)}$, which substantially depends on smoothness $r$. The larger the smoothness -- the faster the error decay. In this section 
we discuss the case, when we do not impose any of the smoothness assumptions. 
Surprisingly, even in such a situation we can guarantee some rate of decay. Results discussed in this section apply in a very general setting. The following result is proved 
in \cite{VT149} (see also \cite{VT89} for previous results). Consider a dictionary
$$
\Di := \{K(\bx,\cdot), \bx\in \Omega^1\}
$$
and define a Banach space $X(K,q)$ as the $L_{q}(\Omega^2)$-closure of span of $\Di$. 

\begin{Theorem}\label{TA.1} Let { $\bW^K_p$} be a class of functions defined above in Section \ref{I}. Assume that { $K\in \mathcal K_{p'}$} satisfies the condition
$$
 { \|K(\bx,\cdot)\|_{L_{p'}(\Omega^2)} \le 1, \quad \bx\in \Omega^1,\quad |\Omega^1|=1}
$$
and { $J_K\in X(K,p')$}. Then for any { $m$} there exists (provided by an appropriate greedy algorithm)  a cubature formula { $Q_m(\cdot,\xi)$} such that
$$
 Q_m(\bW^K_p,\xi)\le C(p-1)^{-1/2} m^{-1/2}, \quad 1< p\le 2 .
$$
\end{Theorem}
As a direct corollary of Theorem \ref{TA.1} and relation (\ref{1.7}) we obtain the following result about the $(K,q)-discrepancy$. 
  
\begin{Theorem}\label{TA.2}   Assume that { $K\in \mathcal K_{q}$} satisfies the condition
$$
 { \|K(\bx,\cdot)\|_{L_{q}(\Omega^2)} \le 1, \quad \bx\in \Omega^1,\quad |\Omega^1|=1}
$$
and  $J_K\in X(K,q)$. Then for any  $m$ there exists (provided by an appropriate greedy algorithm)  a cubature formula  $Q_m(\cdot,\xi)$  such that
$$
 D(\xi,Q,K,q)\le Cq^{1/2} m^{-1/2}, \quad 2\le q <\infty.
$$
\end{Theorem}

\begin{Remark}\label{RA.1} In Theorems \ref{TA.1} and \ref{TA.2} we impose the restriction $1<p\le 2$ or the dual one $2\le q <\infty$. The proof of Theorems \ref{TA.1} and \ref{TA.2} from \cite{VT149} also works in the case $2<p<\infty$ or $1<q<2$ and gives
$$
 Q_m(\bW^K_p,\xi)\le C  m^{-1/p}, \quad 2< p<\infty ,
$$
$$
 D(\xi,Q,K,q)\le C m^{\frac{1}{q}-1}, \quad 1< q <2.
$$
\end{Remark}

Let us discuss a special case $K(\bx,\by)=F(\bx-\by)$, $\Omega^1=\Omega^2 = [0,1)^d$ and $1$-periodic in each variable functions. Then we associate with a cubature formula $(\xi,\La)$ and the function $F$ the function $g_{\xi,\La,F}(\bx)$.
The following Proposition is proved in \cite{VT89}.
\begin{Proposition}\label{PA.1} Let $1<p<\infty$ and $\|F\|_{p'} \le 1$. Then the kernel
$K(\bx,\by)=F(\bx-\by)$ satisfies the assumptions of Theorem \ref{TA.1}.
\end{Proposition}
Proposition \ref{PA.1}, Theorem \ref{TA.1}, Remark \ref{RA.1}, and relation (\ref{1.8}) imply
\begin{Theorem}\label{TA.3} Let $1<p<\infty$ and $\|F\|_{p} \le 1$. Then there exists 
a set $\xi$ of $m$ points such that
$$
\|g_{\xi,Q,F}(\bx)\|_p \le Cp^{1/2} m^{-1/2}, \quad 2\le p<\infty, 
$$
$$
\|g_{\xi,Q,F}(\bx)\|_p \le C m^{\frac{1}{p}-1}, \quad 1< p <2.
$$
\end{Theorem}

Here is a corollary of Theorem \ref{TA.2} and Proposition \ref{PA.1}. 
Let $E\subset [0,1)^d$ be a measurable set. Consider $F(\bx):=\tilde \chi_E(\bx)$. 
\begin{Theorem}\label{TA.4} For any $p\in [2,\infty)$ there exists a set of $m$ points $\xi$ such that  
$$
Q_m(\{\chi_E(\bx-\bz),\bz\in [0,1)^d\},\xi,p) \le Cp^{1/2} m^{-1/2}.
$$
\end{Theorem}
We note that there are interesting results on the behavior of $Q_m(\{\chi_E(\bx-\bz),\bz\in [0,1)^d\},\xi,\infty)$ under assumption that $E$ is a convex set (see \cite{BC}). 
Theorem \ref{TA.4} shows that for $p<\infty$ we do not need any assumptions on the geometry of $E$ in order to get the upper bound $\ll m^{-1/2}$ for the discrepancy. 

The proof of the above Theorems \ref{TA.1}--\ref{TA.4} is constructive (see \cite{VT149}), it is based on the greedy algorithms. We formulate the related result from the theory of greedy approximation.   We remind some notations from the theory of greedy approximation in Banach spaces. The reader can find a systematic presentation of this theory in \cite{Tbook}, Chapter 6. 
Let $X$ be a Banach space with norm $\|\cdot\|$. We say that a set of elements (functions) ${\mathcal D}$ from $X$ is a dictionary   if each $g\in {\mathcal D}$ has norm less than or equal to one ($\|g\|\le 1$)
 and the closure of $\sp {\mathcal D}$ coincides with $X$.  
 
For an element $f\in X$ we denote by $F_f$ a norming (peak) functional for $f$: 
$$
\|F_f\| =1,\qquad F_f(f) =\|f\|.
$$
The existence of such a functional is guaranteed by the Hahn-Banach theorem.

We proceed to  the Incremental Greedy Algorithm (see \cite{T12} and \cite{Tbook}, Chapter 6).       Let $\ep=\{\ep_n\}_{n=1}^\infty $, $\ep_n> 0$, $n=1,2,\dots$ . For a Banach space $X$ and a dictionary $\Di$ define the following algorithm IA($\ep$) $:=$ IA($\ep,X,\Di$).

 {\bf Incremental Algorithm with schedule $\ep$ (IA($\ep,X,\Di$)).} 
  Denote $f_0^{i,\ep}:= f$ and $G_0^{i,\ep} :=0$. Then, for each $m\ge 1$ we have the following inductive definition.

(1) $\ff_m^{i,\ep} \in \Di$ is any element satisfying
$$
F_{f_{m-1}^{i,\ep}}(\ff_m^{i,\ep}-f) \ge -\ep_m.
$$

(2) Define
$$
G_m^{i,\ep}:= (1-1/m)G_{m-1}^{i,\ep} +\ff_m^{i,\ep}/m.
$$

(3) Let
$$
f_m^{i,\ep} := f- G_m^{i,\ep}.
$$

We consider here approximation in uniformly smooth Banach spaces. For a Banach space $X$ we define the modulus of smoothness
$$
\rho(u) := \sup_{\|x\|=\|y\|=1}\left(\frac{1}{2}(\|x+uy\|+\|x-uy\|)-1\right).
$$
It is well known (see for instance \cite{DGDS}, Lemma B.1) that in the case $X=L_p$, 
$1\le p < \infty$ we have
\begin{equation}\label{A.1}
\rho(u) \le \begin{cases} u^p/p & \text{if}\quad 1\le p\le 2 ,\\
(p-1)u^2/2 & \text{if}\quad 2\le p<\infty. \end{cases}     
\end{equation}
 
 Denote by $A_1({\mathcal D}):=A_1(\Di,X)$ the closure in $X$ of the convex hull of ${\mathcal D}$. Proof of Theorem \ref{TA.1} and Remark \ref{RA.1} is based on the following theorem proved in \cite{T12} (see also \cite{Tbook}, Chapter 6).

\begin{Theorem}\label{TA.5} Let $X$ be a Banach space with  modulus of smoothness $\rho(u)\le \gamma u^q$, $1<q\le 2$. Set
$$
\ep_n := \bt\gamma ^{1/q}n^{-1/p},\qquad p:=\frac{q}{q-1},\quad n=1,2,\dots .
$$
Then, for every $f\in A_1({\mathcal D})$ we have
$$
\|f_m^{i,\ep}\| \le C(\bt) \gamma^{1/q}m^{-1/p},\qquad m=1,2\dots.
$$
\end{Theorem}

\section{Discussion}
\label{D}

The first result on the lower bound for discrepancy was the following conjecture of van der Corput \cite{Co1} and \cite{Co2} formulated in 1935. Let $\xi^j\in[0,1]$, $j=1,2,\dots$, then we have
$$
\limsup_{m\to \infty}mD_\infty(\xi^1,\dots,\xi^m)  =\infty.
$$
This conjecture was proved by van Aardenne-Ehrenfest \cite{AE} in 1945:
$$
\limsup_{m\to \infty}\frac{\log\log\log m}{\log\log m}mD_\infty(\xi^1,\dots,\xi^m) >0.
$$
Let us denote
$$
D(m,d)_q :=\inf_{\xi} D_q(\xi),\quad \xi=\{\xi_j\}_{j=1}^m,\quad 1\le q\le \infty.
$$
In 1954 K. Roth \cite{Ro} proved that
\be\label{5.1}
D(m,d)_2 \ge C(d)m^{-1}(\log m)^{(d-1)/2}. 
\ee
In 1972 W. Schmidt \cite{Sch1} proved
\be\label{5.2}
D(m,2)_\infty \ge Cm^{-1}\log m . 
\ee
In 1977 W. Schmidt \cite{Sch} proved
\be\label{5.3}
D(m,d)_q \ge C(d,q)m^{-1}(\log m)^{(d-1)/2},\qquad 1<q\le \infty.  
\ee
In 1981 G. Hal{\' a}sz \cite{Ha} proved
\be\label{5.4}
D(m,d)_1 \ge C(d)m^{-1}(\log m)^{1/2}. 
\ee
The following conjecture has been formulated in \cite{BC} as an excruciatingly difficult great open problem.
\begin{Conjecture}\label{Con5.1} We have for $d\ge 3$
$$
D(m,d)_\infty \ge C(d)m^{-1}(\log m)^{d-1}. 
$$
\end{Conjecture}
This problem is still open. Recently, D. Bilyk and M. Lacey \cite{BL} and D. Bilyk, M. Lacey, and A. Vagharshakyan \cite{BLV} proved
$$
D(m,d)_\infty \ge C(d)m^{-1}(\log m)^{(d-1)/2 + \delta(d)} 
$$
with some positive $\delta(d)$. 

 We now present the results on the lower estimates for the $r$-discrepancy. We
 denote
$$
D^r_q(m,d) := \inf_{\xi}D^r_q(\xi,(1/m,\dots,1/m))
$$
where $D^r_q(\xi,\Lambda)$ is defined in (\ref{2.1}) and also denote
$$
 D^{r,o}_q(m,d) := \inf_{\xi,\Lambda}D^r_q(\xi,\Lambda).
$$
It is clear that 
$$
 D^{r,o}_q(m,d)\le D^r_q(m,d).
$$
The first result on estimating the $r$-discrepancy was obtained in 1985 by V.A. Bykovskii \cite{By}
\be\label{5.5}
D_2^{r,o}(m,d) \ge C(r,d)m^{-r}(\log m)^{(d-1)/2}. 
\ee
This result is a generalization of Roth's result (\ref{5.1}).
The generalization of Schmidt's result (\ref{5.3}) was obtained by the author in 1990 (see \cite{VT43})  
\be\label{5.6}
D^{r,o}_q(m,d) \ge C(r,d,q)m^{-r}(\log m)^{(d-1)/2}, \qquad 1<q\le \infty. 
\ee
In 1994 (see \cite{VT50}) the author proved the lower bounds in the case of weights $\Lambda$ satisfying an extra condition (\ref{1.2}).
\begin{Theorem}\label{T5.1} Let $B$ be a positive number. For any points $\xi^1,\dots,\xi^m \subset \Omega_d$ and any weights $\Lambda =(\lambda_1,\dots,\lambda_m)$ satisfying the condition
\be\label{5.7}
\sum_{\mu=1}^m|\lambda_\mu| \le B
\ee
we have for even integers $r$
$$
D^r_\infty(\xi,\Lambda) \ge C(d,B,r)m^{-r}(\log m)^{d-1}
$$
with a positive constant $C(d,B,r)$.
\end{Theorem}
This result encouraged us to formulate the following generalization of the Conjecture \ref{Con5.1} (see \cite{VT89}).
\begin{Conjecture}\label{Con5.2} For all $d,r\in \N$ we have
$$
D^{r,o}_\infty(m,d) \ge C(r,d)m^{-r}(\log m)^{d-1}. 
$$
\end{Conjecture}

We now proceed to the $r$-smooth $L_\bp$-discrepancy. The first lower bound for such discrepancy was obtained in the case $\bp=\infty$ under an extra condition (\ref{5.7}) on the weights (see \cite{VT165}). Here is the corresponding result from \cite{VT165}.
\begin{Theorem}\label{T5.2}   For any points $\xi^1,\dots,\xi^m \subset \Omega_d$ and  weights $\Lambda =(\lambda_1,\dots,\lambda_m)$ satisfying condition (\ref{5.7})
we have for even integers $r$
$$
\tilde D^r_\infty(\xi,\Lambda) \ge C(d,B,r)m^{-r}(\log m)^{d-1}
$$
with a positive constant $C(d,B,r)$.
\end{Theorem}
Denote as above
$$
 \tilde D^{r,o}_\bp(m,d) := \inf_{\xi,\Lambda}\tilde D^r_\bp(\xi,\Lambda).
$$
Theorem \ref{T5.2} supports the following conjecture.
\begin{Conjecture}\label{Con5.3} For all $d,r\in \N$ we have
$$
\tilde D^{r,o}_\infty(m,d) \ge C(r,d)m^{-r}(\log m)^{d-1}. 
$$
\end{Conjecture}

Theorem \ref{T3.1} gives the following lower bound for $r\in\N$ and $\bp\ge \mathbf 2$
\be\label{5.8}
 \tilde D^{r,o}_\bp(m,d) \ge C(r,d)m^{-r}(\log m)^{(d-1)/2}. 
\ee
The lower bound (\ref{5.8}) is different from the lower bound from Theorem \ref{T5.2}. 
However, Proposition \ref{P3.1} shows that this bound is sharp in case $\bp=\mathbf 2$.
 
Under stronger assumption on $r$, namely, assuming that $r$ is an even number, we obtain a stronger than (\ref{5.8}) lower bound. Theorem \ref{T4.2} gives that 
then for any   $1<p<\infty$
$$
    \tilde D^{r,o}_{p,1}(m,d) \geq   C(r,d,p)   m^{-r}(\log   m)^{(d-1)/2},   \qquad
C(r,d,p)>0.
$$

Proposition \ref{P4.1} shows that the above lower bound is sharp. Moreover, it shows that for $r$ even we have for all $1<p_1<\infty$ and $1\le p_2 \le \infty$
$$
\tilde D^{r,o}_{p_1,p_2}(m,d) \asymp       m^{-r}(\log   m)^{(d-1)/2} .
$$

\end{document}